\documentclass[preprint,12pt]{elsarticle}

\usepackage{amssymb}
\usepackage{amsmath}
\usepackage{amsthm}

\newcommand{\Ff}{\mathbb{F}_}
\newcommand{\Pp}{\mathbb{P}}
\newcommand{\Aa}{\mathbb{A}}

\newcommand{\Cc}{\mathbb{C}}
\newcommand{\ov}{\overline}

\newtheorem{thm}{Theorem}
\newtheorem{prop}{Proposition}
\newtheorem{remark}{Remark}
\newtheorem{conj}{Conjecture}
\newtheorem{lemma}{Lemma}
\journal{}

\begin{document}

\begin{frontmatter}



\title{A new proof of the Carlitz-Wan conjecture on exceptional polynomials}


\author{Yilong Hu\footnote{Corresponding author.\\ email address: huyl10@sjtu.edu.cn}, Zhiyao Zhang} %% Author name

\affiliation{organization={Shanghai Jiao Tong University},%Department and Organization
            addressline={800 Dongchuan Rd.}, 
            city={Shanghai},
            postcode={200240},
            country={China}}

%% Abstract
\begin{abstract}
We give a new proof of the Carlitz-Wan conjecture, previously proved by Lenstra (1995). The proof we present uses Weil’s Conjecture for curves over finite fields and a result of Bombieri and Katz. 
\end{abstract}

%% Keywords
\begin{keyword}
Carlitz-Wan conjecture, exceptional polynomials
\end{keyword}

\end{frontmatter}



\section{Introduction}
\label{sec1}
Exceptional polynomials, as a generalization of permutation polynomials over finite fields, play a crucial role in coding theory and cryptography. Let $\mathbb{F}_q$ be the finite field of $q$ elements with characteristic $p$, and $f$ be a polynomial over $\mathbb{F}_q$. The polynomial $f$ is called an exceptional polynomial if $f$ is a permutation function on infinitely many finite extensions of $\mathbb{F}_q$.

%A widely used characterization of exceptional polynomials is as follows:
%\begin{thm}[Davenport and Lewis \cite{davenport1963notes}]
%    Let $f \in \mathbb{F}_q[x]$ be a polynomial. Define the following bivariate polynomial over $\mathbb{F}_q$:
%    $$F(x,y)=\frac{f(x)-f(y)}{x-y}$$
%    Then $f$ is an exceptional polynomial over $\mathbb{F}_q$ if and only if every irreducible factor of $F(x,y)$ over $\mathbb{F}_q$ can be further decomposed in a finite extension of $\Ff{q}$. In other words, the only absolutely irreducible factors of $f(x)-f(y)$ over $\mathbb{F}_q[x,y]$ are the constant multiples of $x-y$.
%\end{thm}

%Indeed, the property above is the original definition of exceptional polynomials, which was first introduced in \cite{davenport1963notes}. Cohen discussed the permutation property of exceptional polynomials in \cite{cohen1970distribution}, of which the definition in this paper is a direct corollary. It indicates that 

%Exceptional polynomials have a geometric origin: the exceptionality of polynomials is determined by the irreducible decomposition of algebraic curve $f(x)=f(y)$. This inspired us to examine properties of exceptional polynomials through the lens of algebraic curves. It seems that we can prove a conjecture of Carlitz-Wan (now a theorem) by considering algebraic curves of the form: $$f(x)-y^d-y^{d-1}=a.$$

The Carlitz-Wan conjecture provides important information on the degrees of exceptional polynomials. In 1966, Carlitz \cite{carlitz1966number} made a conjecture equivalent to the following proposition: there is no exceptional polynomial of even degree over $\Ff{q}$ with $q$ odd. In 1993, the conjecture was generalized by Wan \cite{wan1993generalization}.

\begin{conj}[Carlitz-Wan]
    If $f$ is an exceptional polynomial over $\mathbb{F}_q$ with degree $d$, then $d$ is coprime to $q-1$.
\end{conj}

There have been several proofs of the Carlitz conjecture and its generalization so far. In \cite{fried1993schur}, Fried, Guralnick and Saxl presented a very detailed analysis of the monodromy groups of exceptional polynomials, which led them to a proof of Carlitz's conjecture. Only one year later, Lenstra discovered a simple and elementary proof of the Carlitz-Wan conjecture, which does not rely on primitive group theory or classification of finite simple groups \cite{cohen1995lenstra}. The arguments made more detailed use of the ramification groups at infinity of the polynomials $f(x)-t$ over $\Ff{q}(t)$. Lenstra presented four different versions of his proof in different settings; however, his original proof was not published, only given informally in his lectures. Two other proofs, using the same general strategy, appeared in the appendix of the paper \cite{guralnick1997exceptional} by Guralnick and M\"uller. Quite recently, the paper \cite{ding2025exceptional} by Ding, Xiong and Zhang also gave quick proofs of the conjecture.

The main purpose of this paper is to give a new proof of the Carlitz-Wan conjecture. The strategy is different from any of the proofs mentioned above. The proof we present below relies on Weil's Conjecture for curves over finite fields and a result of Bombieri and Katz. The remainder of this paper is organized as follows: we will introduce a basic fact of exceptional polynomials, basic notions of algebraic geometry and Bombieri and Katz's result in Section 2, and complete the proof in Section 3.

\section{Preliminaries}

\subsection{Basic facts of exceptional polynomials}

Apart from the definition of exceptional polynomials, we rarely make use of other properties of exceptional polynomials in our proof. That said, there is still one minor fact that we need. It can be found in Michael Zieve's survey article in \textit{Handbook of Finite Fields} \cite[8.4]{mullen2013handbook}, together with various results concerning the classification of exceptional polynomials.

\begin{prop}
\label{111}
    Let $f$ be an exceptional polynomial over $\Ff{q}$. Then there exists a positive integer $M$ such that $f$ permutes $\Ff{q^m}$ when $(m,M)=1$.
\end{prop}

As a corollary, there exists an infinite tower of extensions of finite fields $\Ff{q} \subseteq k_1 \subseteq k_2 \subseteq \ldots$ such that $f$ permutes $k_i$ for every $i$.

\subsection{Basic notions of algebraic geometry}

Let $K$ be an algebraically closed field. We denote by $\Aa^2(K)$ the affine space $K^2$, and by $\Pp^2(K)$ the projective plane over $K$. Both spaces are endowed with their respective Zariski topologies over $K$, where a closed set is the zero locus of a set of polynomials in $K[x,y]$, or of a set of homogeneous polynomials in $K[x_0,x_1,x_2]$. The space $\Pp^2(K)$ (with coordinates $x_0,x_1,x_2$) can be covered with standard open sets $\{x_i \neq 0\}$ $(i=0,1,2)$, and each of them can be identified with $\Aa^2(K)$. For example, $\{x_0 \neq 0\}$ can be identified with $\Aa^2(K)$ via the map $(x_0:x_1:x_2) \mapsto (x_{1}/x_0,x_2/x_0)$.

An affine curve is the zero locus of a single polynomial $f = 0$ in $\Aa^2(K)$, where $f \in K[x,y]$. This curve is \textit{smooth} if and only if $f, \frac{\partial f}{\partial x}$ and $\frac{\partial f}{\partial y}$ do not share common zeros in $\Aa^2(K)$. Similarly, a projective curve is the zero locus of a homogeneous polynomial $f=0$ in $K[x_0,x_1,x_2]$, and it is \textit{smooth} if and only if it is smooth on each of the standard open sets. Due to B\'ezout's theorem, a smooth projective curve is always irreducible.

\subsection{Lower bound for $\#X(\Ff{q})-q-1$}

Let $X$ be a projective smooth curve over $\Ff{q}$ defined by a single homogeneous polynomial $f \in \Ff{q}[x,y,z]$. Denote by $X(\Ff{q})$ the set of $\Ff{q}$-rational points on $X$, namely, the set of solutions to $f = 0$ with coordinates in $\Ff{q}$. The celebrated Weil Conjecture for curves states the following:

\begin{thm}[Weil]
    Let the genus of $X$ be $g$. Then there exist $2g$ algebraic integers $\alpha_1,\alpha_2,\ldots,\alpha_{2g} \in \Cc$ such that $$\#X(\Ff{q^n}) = q^n+1 - \sum_{i=1}^{2g} \alpha_i^{n}.$$ Furthermore, all $\alpha_i$ are $q$-Weil numbers, i.e. as a complex number, any Galois conjugate of $\alpha_i$ has complex modulus $q^{1/2}$.
\end{thm}
One wants to measure how far the quantity $\#X(\Ff{q})$ may differ from the main term $q+1$. Since $\alpha_i$ are $q$-Weil numbers, we have: $$\lvert \#X(\Ff{q})-q-1 \rvert \leq 2gq^{1/2}.$$ This is the famous Hasse-Weil bound. Intuitively, this bound shows that the quantity $\#X(\Ff{q})$ should be reasonably close to $q+1$. Additionally, for a fixed $X$, $\#X(\Ff{q^t})$ should not stay too close to $q^t+1$ either (as $t \to \infty$), unless $\#X(\Ff{q^t})=q^t+1$: \begin{thm}[Bombieri and Katz \cite{bombieri2010note}]
    Let $X$ be a projective smooth curve over $\Ff{q}$. Define a sequence of integers: $$A(n) = \sum_{i=1}^{2g} \alpha_i^n$$ so that $\#X(\Ff{q^n})=q^n+1-A(n)$. Then for any $M > 0$, there exists a positive integer $N$ such that when $n>N$, either $A(n)= 0$ or $\lvert A(n) \rvert > M$.
\end{thm}

\begin{remark}
    It could very well be the case that $\#X(\Ff{q^t})=q^t+1$ for infinitely many extensions $\Ff{q^t}$, i.e. $A(t)$ takes zero value infinitely many times. A straightforward example is the projective line $\Pp^1$ over $\Ff{q}$. A less trivial example is an exceptional cover of $\Pp^1$ \cite{fried1993schur}. This does not violate the theorem of Bombieri and Katz since it only asks for nonzero values of $A(n)$ to go to infinity.

    Moreover, if there are infinitely many zeros of $A(n)$, then the Skolem-Mahler-Lech theorem implies that the zeros must
    fall into a union of finitely many arithmetic progressions (apart from finitely many exceptions).
\end{remark}

\section{Proof of Conjecture}

Assume that an exceptional polynomial $f \in \Ff{q}[x]$ of degree $d$ exists, such that $(d,q-1) \neq 1$. By replacing $f(x)$ with $f(x)-f(0)$ we may assume that $f(0)=0$. A polynomial $f(x)$ is exceptional if and only if its $p$-th power $f(x)^p$ is exceptional. Without loss of generality, we require that there does not exist $g$ such that $f(x)=g(x)^p$. Otherwise, we may replace $f(x)$ with $g(x)$. Define a bivariate homogeneous polynomial of degree $d$ in $\Ff{q}[x,z]$ by $F(x,z) = z^d f(x/z)$. Due to Proposition \ref{111}, we observe that one may replace $\Ff{q}$ with a large extension $k$ such that $f$ is still exceptional on $k$. Note that the property $(d,q-1) \neq 1$ still holds after replacing $\Ff{q}$ by its finite extension. Next, we state and prove the following lemma.

\begin{lemma}
    For $q$ sufficiently large, there exist $0 \neq a,b \in \Ff{q}$, such that the projective curve defined in $\mathbb{P}^2(\ov{\Ff{q}})$ (with coordinates $x,y,z$) $$X: G(x,y,z) =0.$$ is smooth, where $G(x,y,z) = F(x,z)-y^d-b \cdot y^{d-1}z - az^d$.
\end{lemma}

\begin{proof}
    Set-theoretically, $X$ can be split into two parts, the finite part $X_0:= X \cap \{z \neq 0\}$ and the points at infinity $X_1:=X \cap \{z =0\}$. One only needs to prove the curve can be chosen to be smooth on both sets. 

    To begin with, we consider the points at infinity. In this case, the points at infinity are $\{x^d-y^d=0,z=0\}$, which can be regarded as a subset of the projective line over $\ov{\Ff{q}}$. Note that there exist no points on $X$ such that $x=z=0$, so one only needs to focus on the standard open set $\{x\neq 0\}$. The affine curve $X \cap \{x \neq 0\}$ is defined by $G(1,y,z)=0$, and the partial derivative of $G(1,y,z)$ with respect to $z$ is equal to $F_z(1,z) - (ad)z^{d-1} - y^{d-1}\cdot b$. Substitute $z=0$, then the expression becomes $F_z(1,0)-y^{d-1}\cdot b$. Note that $y^d=1$ on $X_1$, so when $d>q$, one can choose some $b \in \Ff{q}$ such that the partial derivative of $G(1,y,z)$ with respect to $z$ is non-zero on every point of $X_1$. 
    
    Next, with the value of $b$ fixed, we turn our attention to the finite part. In this part, the curve $X_0=X \cap \{z \neq 0\}$ is the affine curve $G(x,y,1)=0$, and the partial derivatives of $G(x,y,1)$ with respect to $x, y$ equal to $(f'(x),-dy^{d-1}-b(d-1)y^{d-2})$. Let $y_0:=\frac{-b(d-1)}{d}$ and $y_1:=y_0^{d-1}(y_0+b)$. Since the value of $b$ is fixed, $y_0,y_1$ are also fixed. If the partial derivative with respect to $y$ is zero somewhere on $X_0$, then $y=0$ or $y=y_0$. In this case, either $f(x)=a$ or $f(x)=a+y_1$. To make $X_0$ smooth, all we need is to choose $a$ properly so that the partial derivative with respect to $x$ is non-zero when $f(x)=a$ or $f(x)=a+y_1$. This can always be done when $q$ is sufficiently large, since $f'(x)$ has at most $(d-1)$ zeros in $\ov{\Ff{q}}$. 
\end{proof}

Passing to a finite extension of $\Ff{q}$ if necessary, we assume a smooth $X$ exists. Due to Proposition \ref{111}, there exists a tower of finite extensions $\Ff{q} \subseteq k_1 \subseteq k_2 \subseteq \ldots$ such that $f$ is a permutation polynomial over each $k_i$. Note that $(d,\#k_i-1)\neq 1$. On the finite part $X_0$ of $X$, clearly $\#X_0(k_i) = \#k_i$, since $f$ permutes $k_i$. Indeed, for every fixed $y$, there exists a unique $x$ such that $f(x)=y^d-by^{d-1}+a$.

On the other hand, since $(d,\#k_i-1) \neq 1$, we have $\#(\mu_d \cap k_i) > 1$, where $\mu_d$ is the group of $d$-th roots of unity in $\ov{\Ff{q}}$. Elements of $X_1(k_i)$ have the form $(1:y:0)$, where $y^d=1$. Therefore, we have $1 < \#X_1(k_i) \leq d$ for every $i$.

Write $k_i = \Ff{q^{n_i}}$. Since
$$\#X(k_i) = \#X_0(k_i)+ \#X_1(k_i) = \#k_i+1+A(n_i).$$ We have $A(n_i) = \#X_1(k_i)-1$ nonzero and bounded by $d$. As $n_i \to \infty$, this contradicts Bombieri and Katz's results. QED

\section*{Funding sources}

This research did not receive any speciﬁc grant from funding agencies in the public, commercial, or not-for-proﬁt sectors.
\section*{Acknowledgement}

The work was supported in part by the National Key Research and Development Program under Grant 2022YFA1004900 and 2025YFA1017202. The authors thank Michael Zieve, Daqing Wan and anonymous reviewers for helping us fix the error and offering valuable suggestions in the preparation of our article. The authors also thank Zeyu Lu and Jiaming Zhang for discussions.

\bibliographystyle{elsarticle-num} 
\bibliography{name}




%% For numbered reference style

%%\bibitem{lamport94}
%% Leslie Lamport,
%% \textit{\LaTeX: a document preparation system},
%% Addison Wesley, Massachusetts,
%% 2nd edition,
%% 1994.

%%\end{thebibliography}
\end{document}